\documentclass[12pt, reqno]{amsart}
\usepackage{ amsmath,amsthm, amscd, amsfonts, amssymb, graphicx, color}
\usepackage[bookmarksnumbered, colorlinks, plainpages]{hyperref}
\newtheorem{thm}{Theorem}[section]
\newtheorem{cor}[thm]{Corollary}
\newtheorem{lem}[thm]{Lemma}

\newtheorem{exam}[thm]{Example}
\numberwithin{equation}{section}

\begin{document}

\title{generalized Jacobson's Lemma in a Banach algebra}

\author{Huanyin Chen}
\author{Marjan Sheibani}
\address{
Department of Mathematics\\ Hangzhou Normal University\\ Hang -zhou, China}
\email{<huanyinchen@aliyun.com>}
\address{Women's University of Semnan (Farzanegan), Semnan, Iran}
\email{<sheibani@fgusem.ac.ir>}


\begin{abstract}
Let $\mathcal{A}$ be a Banach algebra, and let $a,b,c\in \mathcal{A}$ satisfying $$a(ba)^2=abaca=acaba=(ac)^2a.$$ We prove that $1-ba\in \mathcal{A}^d$ if and only if $1-ac\in \mathcal{A}^d$. In this case, $(1-ac)^{d}=\big[1-a(1-ba)^{\pi}(1-\alpha (1+ba))^{-1}bac\big](1+ac)+a(1-ba)^dbac$. This extends the main result on g-Drazin inverse of Corach (Comm. Algebra, {\bf 41}(2013), 520--531).\end{abstract}

\maketitle

\section{Introduction}

Let $\mathcal{A}$ be a complex Banach algebra with identity. An element $a\in \mathcal{A}$ has g-Drazin inverse in case there exists $x\in \mathcal{A}$
such that $$x=xax, ax=xa, a-a^2x\in \mathcal{A}^{qnil}.$$ The preceding $x$ is unique if it exists, we denote it by $a^{d}$. Here, $\mathcal{A}^{qnil}=\{a\in \mathcal{A}~|~1+ax\in \mathcal{A}^{-1}~\mbox{whenever}~ax=xa\}$, where $\mathcal{A}^{-1}$ stands for the set of all invertible elements of $\mathcal{A}$.

For any $a,b\in \mathcal{A}$, Jacobson's Lemma for invertibility states that $1-ab\in \mathcal{A}^{-1}$ if and only if $1-ba\in \mathcal{A}^{-1}$ and $(1-ba)^{-1}=1+b(1-ab)^{-1}a$ (see~\cite[Lemma 1.4]{M}). Jacobson's Lemma plays an important role in matrix and operator theory. Let $a,b\in \mathcal{A}$. Zhuang et al. proved the Jacobson's Lemma for g-Drazin inverse. That is, it was proved that $1-ab\in \mathcal{A}^d$ if and only if $1-ba\in \mathcal{A}^d$ and $$(1-ba)^d=1+b(1-ab)^da$$ (see~\cite[Theorem 2.3]{Z}). Corach et al. generalized Jacobson's Lemma for g-Drain inverse to the case that $aba=aca$ (see~\cite[Theorem 1]{C}).

The motivation of this paper is to present a new generalized Jacobson's lemma for generalized Drazin inverses. We thereby extend ~\cite[Theorem 2.3]{Z} to a wider case.
Let $a,b\in \mathcal{A}$. In ~\cite[Theorem 2.2]{C1}, the authors considered the condition $$a(ba)^2=abaca=acaba=(ac)^2a$$ for Cline's formula, i.e., $ba\in \mathcal{A}^d$ if and only if $ac\in \mathcal{A}^d$ and $(ba)^d=b[(ac)^d]^2a$. Common local spectral properties for bounded linear operators under the preceding conditions was investigated in ~\cite{ZG}. This raises a problem if Jacobson's lemma for generalized Drazin inverse hold under such wider condition. We shall give a confirmative answer to this problem. That is, If $a,b,c,\in \mathcal{A}$ satisfying $a(ba)^2=abaca=acaba=(ac)^2a,$ we prove that $1-ba\in \mathcal{A}^d$ if and only if $\beta=1-ac\in \mathcal{A}^d$. In this case, $(1-ac)^{d}=\big[1-d(-ba)^{\pi}(1-\alpha (1+ba))^{-1}bac\big](1+ac)+a(1-ba)^dbac.$

The Drazin inverse of $a\in \mathcal{A}$, denoted by $a^D$, is the unique element $a^D$ satisfying the following three equations $$a^D=a^Daa^D, aa^D=a^Da, a^k=a^{k+1}a$$ for some $k\in {\Bbb N}$. The small integer $k$ is called the Drazin index of $a$, and is denoted by $i(a)$. Moreover, we prove the generalized Jacobson's lemma for the Drazin inverse under the preceding condition.

Throughout the paper, all Banach algebra are complex with identity. $\mathcal{A}^{D}$ and $\mathcal{A}^{d}$ denote the sets of all Drazin and g-Drazin invertible elements in $\mathcal{A}$ respectively. We use $\mathcal{A}^{nil}$ to denote the set of all nilpotents of Banach algebra $\mathcal{A}$. ${\Bbb C}$ stands for the field of all complex numbers.

\section{generalized Jacobson's lemma}

We come now to the main result of this paper which will be the tool in our following development.

\begin{lem} Let $\mathcal{A}$ be a Banach algebra, let $m\in {\Bbb N}$ and let $a\in \mathcal{A}$. Then $a$ has g-Drazin inverse if and only if there exists $b\in comm(a)$ such that
$$\big[ab-(ab)^2\big]^m=0, a-a^2b\in \mathcal{A}^{qnil}.$$ In this case, $$a^d=(a+1-e)^{-1}e, e=\sum\limits_{i=0}^{m-1}(ab)^{2m-i}(1-ab)^{i}.$$\end{lem}
\begin{proof} $\Longrightarrow $ Since $a\in \mathcal{A}^d$, there exists $b\in comm(a)$ such that $$b=bab, a-a^2b\in \mathcal{A}^{qnil}.$$ Hence $ab=(ab)^2$.
Moreover, we have $e=\sum\limits_{i=0}^{m-1}\left(
\begin{array}{c}
2m\\
i
\end{array}
\right)(ab)^{2m-i}(1-ab)^{i}=ab$, and then $(a+1-e)^{-1}e=(a+1-ab)^{-1}ab=(a+1-ab)^{-1}(a+1-ab)b=b=a^d$, as desired.

$\Longleftarrow $ Let $$e=\sum\limits_{i=0}^{m-1}\left(
\begin{array}{c}
2m\\
i
\end{array}
\right)(ab)^{2m-i}(1-ab)^{i}, f=\sum\limits_{i=0}^{m-1}
\left(
\begin{array}{c}
2m\\
i
\end{array}
\right)(ab)^{2m-i}(1-ab)^{i}.$$ Since $\big[ab-(ab)^2\big]^{m}=0$, we have
$(ab)^m(1-ab)^m=0$, and so $$e+f=[ab+(1-ab)]^{2m}=1, ef=fe=0.$$ Hence, $e^2=e^2+ef=e$.
Clearly, $a,b$ and $e$ commute one another. Then we have $$\begin{array}{lll}
w:&=&ab-e\\
&=&\big[ab-(ab)^2\big]+(ab)\big[ab-(ab)^2\big]+\cdots +(ab)^{2m-2}\big[ab-(ab)^2\big]\\
&-&\sum\limits_{i=1}^{m-1}\sum\limits_{i=0}^{m-1}
\left(
\begin{array}{c}
2m\\
i
\end{array}
\right)(ab)^{2m-i}(1-ab)^{i}.
\end{array}$$
Hence $w^m=0$. By hypothesis, $a-a^2b=a(1-ab)=a(1-e)-aw\in \mathcal{A}^{qnil}$, and so $a(1-e)\in \mathcal{A}^{qnil}$.
Let $c=bab$. Then $$\begin{array}{lll}
ac+1-e&=&1+(ab)^2-(ab)^{2m}-\sum\limits_{i=1}^{m-1}(ab)^{2m-i}(1-ab)^{i}\\
&=&1+\big[(ab)^2-(ab)^3\big]+\big[(ab)^3-(ab)^4\big]\\
&+&\cdots + \big[(ab)^{2m-1}-(ab)^{2m}\big]-\sum\limits_{i=1}^{m-1}(ab)^{2m-i}(1-ab)^{i}\\
&\in& \mathcal{A}^{-1}.
\end{array}$$ Also we have $$c(1-e)=bab\big[(1-(ab)^{2m})-\sum\limits_{i=1}^{m-1}(ab)^{2m-i}(1-ab)^{i}\in \mathcal{A}^{nil}.$$
We easily check that
$$\begin{array}{lll}
(a+1-e)(c+1-e)&=&ac+a(1-e)+(1-e)c+1-e\\
&=&ac+1-e+a(1-e)+c(1-e)\\
&\in &\mathcal{A}^{-1}.
\end{array}$$
Hence $a+(1-e)\in \mathcal{A}^{-1}$.
Therefore $a^d=(a+1-e)^{-1}e,$ as required.\end{proof}

\begin{thm} Let $\mathcal{A}$ be a Banach algebra, and let $a,b,c\in \mathcal{A}$ satisfying $$a(ba)^2=abaca=acaba=(ac)^2a.$$ Then $\alpha=1-ba\in \mathcal{A}^d$ if and only if $\beta=1-ac\in \mathcal{A}^d$. In this case,
$$\beta^{d}=\big[1-a\alpha^{\pi}(1-\alpha (1+ba))^{-1}bac\big](1+ac)+a\alpha^dbac.$$\end{thm}
\begin{proof} Let $p=\alpha^{\pi},x=\alpha^d$. Then $1-p\alpha (1+ba)\in \mathcal{A}^{-1}$. Let $$y=\big[1-ap(1-p\alpha (1+ba))^{-1}bac\big](1+ac)+axbac.$$

Step 1. $\big[y\beta-(y\beta)^2\big]^2=0$. We see that
$$y\beta=1-(ac)^2-ap(1-p\alpha (1+ba))^{-1}bac\big[1-(ac)^2\big]+axbac(1-ac).$$
We compute that
$${\scriptsize\begin{array}{lll}
y\beta a&=&\big[1-(ac)^2-ap(1-p\alpha (1+ba))^{-1}bac\big[1-(ac)^2\big]+axbac(1-ac)\big]a\\
&=&a-(abac-axbac(1-ac))a-ap(1-p\alpha (1+ba))^{-1}[bac-bac(ac)^2\big]a\\
&=&a-[abac-ax(bac-bacac)]a-ap(1-p\alpha (1+ba))^{-1}[ba-bacaba]ca\\
&=&a-[abac-ax(1-ba)bac]a-ap(1-p\alpha (1+ba))^{-1}[1-(ba)^2]baca\\
&=&a-apbaca-ap(1-p\alpha (1+ba))^{-1}p\alpha (1+ba)baca\\
&=&a-ap(1-p\alpha (1+ba))^{-1}\big[(1-p\alpha (1+ba))+p\alpha (1+ba)\big]baca\\
&=&\big[1-ap(1-p\alpha(1+ba))^{-1}bac\big]a.\\
\end{array}}$$
Set $z=cac+p(1-p\alpha (1+ba))^{-1}bac\big[1-(ac)^2\big]+xbac(1-ac).$ Then $y\beta=1-az$, and so
$$\begin{array}{lll}
y\beta (1-y\beta)&=&y\beta az\\
&=&\big[1-ap(1-p\alpha(1+ba))^{-1}bac\big]az\\
&=&1-y\beta-ap(1-p\alpha(1+ba))^{-1}bac(1-y\beta).
\end{array}$$
Since $acaba=a(caba)=a(caca)=(ac)^2a=(abac)a=abaca$, we have $(baca)(ba)=(ba)(baca)$, and so $(baca)\alpha =\alpha (baca).$ Hence,
$(baca)x=x(baca),$ and then
$$ap(1-p\alpha(1+ba))^{-1}bacaxbac=ap(1-p(1+ba))^{-1}(baca)\alpha xbac=0.$$
Moreover, we have
$${\scriptsize\begin{array}{ll}
&ap(1-p\alpha(1+ba))^{-1}bacy\beta(1-y\beta)\\
=&ap(1-p\alpha(1+ba))^{-1}bac\big[1-ap(1-p\alpha (1+ba))^{-1}bac(1+ac)\big]a(1-ca)z\\
=&\big[ap(1-p\alpha(1+ba))^{-1}bac(1+ac)+ap(1-p\alpha(1+ba))^{-2}(baca)bac(1+ac)\big]a(1-ca)z\\
=&\big[ap(1-p\alpha(1+ba))^{-1}(1+ba)bac+ap(1-p\alpha(1+ba))^{-2}(ba)^2(1+ba)bac\big]a(1-ca)z\\
=&\big[ap(1-p\alpha(1+ba))^{-2}\big[p-p\alpha(1+ba)-p(ba)^2\big](1+ba)bac\big]a(1-ca)z\\
=&0.
\end{array}}$$ Hence, $$\begin{array}{lll}
y\beta (1-y\beta)^2&=&\big[1-y\beta-ap(1-p\alpha(1+ba))^{-1}bac(1-y\beta)\big](1-y\beta)\\
&=&\big[1-y\beta-ap(1-p\alpha(1+ba))^{-1}bac\big](1-y\beta)\\
&=&\big[(1-ap(1-p\alpha(1+ba))^{-1}bac)a-y\beta a\big]z\\
&=&0.\end{array}$$ Therefore $\big[y\beta-(y\beta)^2\big]^2=y^2\beta^2(1-y\beta)^2=0$.

Step 2. $\beta-\beta y\beta\in \mathcal{A}^{qnil}$. Clearly, we have
$$baca(baca)^2p\alpha(1-p\alpha(1+ba))^{-3}=(baca)^3(1-p\alpha(1+ba))^{-3}(\alpha-\alpha^2\alpha)\in \mathcal{A}^{qnil}.$$
By using Cline's formula, we derive $$\begin{array}{ll}
&(1-y\beta )a\alpha bacap(1-p\alpha(1+ba))^{-2}bac\\
=&ap(1-p\alpha(1+ba))^{-1}baca\alpha bacap(1-p\alpha(1+ba))^{-2}bac\\
=&a(baca)^2p\alpha(1-p\alpha(1+ba))^{-3}bac\\
\in & \mathcal{A}^{qnil}.
\end{array}$$

By using Cline's formula again, we have
$$\begin{array}{ll}
&\beta (1-y\beta )^3\\
=&\beta ap(1-p\alpha(1+ba))^{-1}(baca)p(1-p\alpha(1+bd))^{-1}bac (1-y\beta )\\
=&\beta a(baca)p(1-p\alpha(1+bd))^{-2}bac(1-y\beta )\\
=&(1-ac)a(baca)p(1-p\alpha(1+ba))^{-2}bac(1-y\beta )\\
=&a\alpha bacap(1-p\alpha(1+ba))^{-2}bac(1-y\beta )\\
\in & \mathcal{A}^{qnil}.
\end{array}$$
 Therefore $\beta (1-y\beta )^3\in \mathcal{A}^{qnil}$, and so $(\beta-\beta^2y)^3\in \mathcal{A}^{qnil}$,
 This implies that $\beta-\beta^2y\in \mathcal{A}^{qnil}$, as required.

Step 3. $y\in comm(\beta)$. Set $s=ac$. Then we have

Claim 1. $\beta(axbac)=(axbac)\beta.$ We easily check that
$$(bacsaba)ba=bacacababa=babacacaba==ba(bacsaba),$$ and so $$(bacsaba)\alpha=\alpha (bacsaba).$$
Hence $(bacsaba)x=x(bacsaba)$.
Obviously, $$p=(ba)^2p[1-p\alpha (1+ba)]^{-1}=(ba)^4p[1-p\alpha (1+ba)]^{-2}.$$ Then
$$\begin{array}{lll}
s(apbac)&=&sa(ba)^4p[1-p\alpha (1+ba)]^{-2}bac\\
&=&s(ac)^2ababap[1-p\alpha (1+ba)]^{-2}bac\\
&=&a(bacsaba)bap[1-p\alpha (1+ba)]^{-2}bac\\
&=&abap[1-p\alpha (1+ba)]^{-2}(bacsaba)bac\\
&=&abap[1-p\alpha (1+ba)]^{-2}bacs(ac)^3\\
&=&abap[1-p\alpha (1+ba)]^{-2}(ba)^3bacs\\
&=&a(ba)^4p[1-p\alpha (1+ba)]^{-2}bacs\\
&=&(apbac)s.
\end{array}$$ Since $sabac=(acaba)c=(abaca)c=abacs$, we have
$$sa\alpha xbac=sabac-s(apbac)=abacs-(apbac)s=a\alpha xbacs,$$ and so
$$saxbac-sabaxbac=axbacs-abaxbacs.$$

On the other hand, we have
$$\begin{array}{lll}
s(abapbac)&=&sa(ba)^5p[1-p\alpha (1+ba)]^{-2}bac\\
&=&s(ac)^4abap[1-p\alpha (1+ba)]^{-2}bac\\
&=&ababa(bacsaba)p[1-p\alpha (1+ba)]^{-2}bac\\
&=&ababap[1-p\alpha (1+ba)]^{-2}(bacsaba)bac\\
&=&ababap[1-p\alpha (1+ba)]^{-2}bacs(ac)^3\\
&=&ababap[1-p\alpha (1+ba)]^{-2}(ba)^3bacs\\
&=&aba(ba)^4p[1-p\alpha (1+ba)]^{-2}bacs\\
&=&(abapbac)s.
\end{array}$$ Since $sababac=s(ac)^3=(ac)^3s=ababacs$, we have
$$aba\alpha xbacs=saba\alpha xbac.$$ Then we have $$\begin{array}{lll}
ababa\alpha xbacs&=&ac(aba\alpha xbacs)\\
&=&ac(saba\alpha xbac)\\
&=&sacaba\alpha xbac\\
&=&sababa\alpha xbac,
\end{array}$$
and so $$aba(1+ba)\alpha xbacs=saba(1+ba)\alpha xbac,$$ and then
$$abaxbacs-aba(ba)^2xbacs=sabaxbac-saba(ba)^2xbac.$$
One easily checks that
$$\begin{array}{lll}
a(ba)^3xbacs&=&ax(ba)^3bacs\\
&=&axb(ac)^4\beta \\
&=&ax(bacsaba)bac\\
&=&a(bacsaba)xbac\\
&=&(ac)^2sabaxbac\\
&=&sa(ba)^3xbac.
\end{array}$$
This implies that $abaxbacs=sabaxbac$, and therefore $\beta (axbac)=(axbac)\beta.$

Claim 2. $sap(1-p\alpha (1+ba))^{-1}bac(1+ac)=ap(1-p\alpha (1+ba))^{-1}bac(1+ac)\beta $.
Set $t=ap(1-p\alpha (1+ba))^{-1}bac(1+ac).$ Since
$(bacsaba)ba=ba(bacsaba)$, we have
$$\begin{array}{lll}
st&=&sap(1-p\alpha (1+ba))^{-1}bac(1+ac)\\
&=&sa(ba)^4p[1-p\alpha (1+ba)]^{-3}bac(1+ac)\\
&=&(ac)^3sabap[1-p\alpha (1+ba)]^{-3}bac(1+ac)\\
&=&a(ba)^4bap[1-p\alpha (1+ba)]^{-3}bac(1+ac)\\
&=&abap[1-p\alpha (1+ba)]^{-3}(ba)^4bac(1+ac)\\
&=&abap[1-p\alpha (1+ba)]^{-3}b(ac)^5(1+ac)
\end{array}$$
Also we have
$$\begin{array}{lll}
ts&=&ap(1-p\alpha (1+ba))^{-1}bac(1+ac)\beta \\
&=&ap[1-p\alpha (1+ba)]^{-3}(ba)^4bsac(1+ac)\\
&=&abap[1-p\alpha (1+ba)]^{-3}b(ac)^3sac(1+ac)\\
&=&abap[1-p\alpha (1+ba)]^{-3}bs(ac)^4(1+ac)\\
\end{array}$$
Hence, $st=ts$, and so $\beta t=t\beta$.

Therefore $y\in comm(\beta)$, and then $y=\beta^d$, as required.

$\Longleftarrow$ Since $1-ac\in R^d$, it follows by Jacobson's Lemma that $1-ca\in R^d$. Applying the preceding discussion, we obtain that $1-ba\in R^d$, as desired.
\end{proof}

\begin{cor} Let $R$ be a Banach algebra, let $\lambda\in {\Bbb C}$, and let $a,b,c\in R$ satisfying $$a(ba)^2=abaca=acaba=(ac)^2a.$$ Then $\lambda-ba\in R^d$ if and only if $\lambda-ac\in R^d$. In this case,
$$\begin{array}{ll}
&(\lambda-ac)^d\\
=&\begin{cases}
-a((ba)^d)^2c& \lambda=0\\
\big[1-a(1-ba)^{\pi}(\lambda^2-(\lambda-ba)(\lambda+ba))^{-1}bac\big]\\
(\lambda+ac)+\frac{1}{\lambda} a(\lambda-ba)^dbac &\lambda\neq 0.
\end{cases}\end{array}$$\end{cor}
\begin{proof} Case 1. $\lambda =0$. It is obvious by~\cite[Theorem 2.2]{C1}.

Case 2. $\lambda\neq 0$. Let $d=\frac{b}{\lambda}$ and $e=\frac{c}{\lambda}$. It is easy to show that $a(da)^2=adaea=aeada=(ae^2)a$. By virtue of Theorem 2.2, $1-da\in R^d$ if and only if $1-ae\in R^d$. That is, $\lambda-ba\in R^d$ if and only if $\lambda-ac\in R^d$. In this case,
$$\begin{array}{ll}
&(\lambda-ac)^d\\
=&\lambda (1-ae)^d\\
=&\lambda \big[1-a(1-da)^{\pi}(1-(1-da)(1+da))^{-1}dae\big]\\
&(1+ae)+a(1-ba)^ddae\\
=&\big[1-a(1-ba)^{\pi}(\lambda^2-(\lambda-ba)(\lambda+ba))^{-1}bac\big]\\
&(\lambda+ac)+\frac{1}{\lambda} a(\lambda-ba)^dbac,
\end{array}$$ as desired.
\end{proof}

\begin{cor} Let $R$ be a Banach algebra, let $\lambda\in {\Bbb C}$, and let $a,b,c\in R$ satisfying $aba=aca$ Then $\lambda-ba\in R^d$ if and only if $\lambda-ac\in R^d$. In this case,
$$\begin{array}{ll}
&(\lambda-ac)^d\\
=&\begin{cases}
-a((ba)^d)^2c& \lambda=0\\
\big[1-a(1-ba)^{\pi}(\lambda^2-(\lambda-ba)(\lambda+ba))^{-1}bac\big]\\
(\lambda+ac)+\frac{1}{\lambda} a(\lambda-ba)^dbac &\lambda\neq 0.
\end{cases}\end{array}$$\end{cor}
\begin{proof} This is obvious by Corollary 2.3.\end{proof}

\section{Drazin inverse}

As it is known, $a\in \mathcal{A}^D$ if and only if there exists $x\in R$ such that $x=xax, x\in comm(a), a-a^2x\in \mathcal{A}^{nil}$, and so $a^D=a^d$. For the generalized Jacobson's Lemma for Drazin inverse, we have

\begin{thm} Let $\mathcal{A}$ be a Banach algebra, and let $a,b,c\in \mathcal{A}$ satisfying $$a(ba)^2=abaca=acaba=(ac)^2a.$$ Then $\alpha=1-ba\in \mathcal{A}^D$ if and only if $\beta=1-ac\in \mathcal{A}^D$. In this case,
$$\begin{array}{ll}
&\beta^{D}\\
=&\big[1-a\alpha^{\pi}(1-\alpha (1+ba))^{-1}bac\big](1+ac)+a\alpha^Dbac,
\end{array}$$
$$i(1-ba)\leq i(1-ac)+1.$$
\end{thm}
\begin{proof} Let $p=\alpha^{\pi},x=\alpha^D$. In view of Theorem 2.1, $\beta\in R^d$ and $$\beta^d=\big[1-ap(1-p\alpha (1+ba))^{-1}bac\big](1+ac)+dxbac.$$ We shall prove that $\beta^D=\beta^d$.

We will suffice to check $\beta-\beta \beta^d\beta\in \mathcal{A}^{nil}$.
As in the proof of Theorem 2.2, we have
$$\begin{array}{lll}
\beta-\beta \beta^d\beta&=&\beta (1-\beta^d\beta )\\
&=&d\alpha bacap(1-p\alpha(1+ba))^{-2}bac.
\end{array}$$ In light of~\cite[Lemma 2.1]{C1}, we will suffice to prove $$baca\alpha bacap(1-p\alpha(1+ba))^{-2}\in \mathcal{A}^{nil}.$$
Similarly to the discussion in Theorem 2.2, we see that $baca\in comm(\alpha)$, and so $baca, \alpha p$ and $(1-p\alpha(1+ba))^{-2}$ commute one another.
Set $n=i(\alpha)$. Then $$\begin{array}{ll}
&\big[baca\alpha bacdp(1-p\alpha(1+ba))^{-2}\big]^n\\
=&(baca)^2(1-p\alpha(1+ba))^{-2n}(\alpha-\alpha^2\alpha^d)^n\\
=&0;
\end{array}$$ hence,
$$\begin{array}{l}
(\beta-\beta \beta^d\beta)^{n+1}\\
=a\alpha bacdp(1-p\alpha(1+ba))^{-2}\big[baca\alpha bacap(1-p\alpha(1+ba))^{-2}\big]^nbac\\
=0.
\end{array}$$ Thus, we have
$\beta-\beta \beta^d\beta\in \mathcal{A}^{nil}$. Moreover, we have $i(\beta)\leq i(\alpha)+1$, as desired.\end{proof}

The group of $a\in \mathcal{A}$ is the unique element $a^{\#}\in \mathcal{A}$ which satisfies $a=aa^{\#}a, a^{\#}=a^{\#}aa^{\#}.$ We denote the set of all group invertible elements of $\mathcal{A}$ by $\mathcal{A}^{\#}$. As is well known, $a\in \mathcal{A}^{\#}$ if and only if $a\in \mathcal{A}^D$ and $i(a)=1$. We are now ready to prove:

\begin{cor} Let $\mathcal{A}$ be a Banach algebra, and let $a,b,c\in \mathcal{A}$ satisfying $$a(ba)^2=abaca=acaba=(ac)^2a.$$ Then $\alpha=1-ba$ has group inverse if and only if $\beta=1-ac$ has group inverse. In this case, $$\beta^{\#}=\big[1-a\alpha^{\pi}(1-\alpha (1+ba))^{-1}bac\big](1+ac)+a\alpha^{\#}bac.$$\end{cor}
\begin{proof} Since $1-ba\in \mathcal{A}^{\#}$, we have $1-ba\in \mathcal{A}^D$. In light of Theorem ???, $1-ac\in \mathcal{A}^D$.
Let $\alpha=1-ba$ and $\beta=1-ac$. Let $p=1-\alpha\alpha^D$. Since $\alpha\in R^{\#}$, we have $\alpha p=\alpha-\alpha^2\alpha^D=0$. As in the proof of Theorem 2.2,
we have $$\begin{array}{l}
\beta-\beta \beta^D\beta\\
=d\alpha bacdp(1-p\alpha(1+ba))^{-2}\big[bacd\alpha bacdp(1-p\alpha(1+ba))^{-2}\big]^nbac\\
=dbacd\alpha p(1-\alpha(1+ba))^{-2}\big[bacd\alpha bacdp(1-p\alpha(1+ba))^{-2}\big]^nbac\\
=0.
\end{array}$$ Obviously, $\beta^D\in comm(\beta)$ and $\beta^D=\beta^D\beta\beta^D$.
Therefore $$\begin{array}{l}
\beta^{\#}=\beta^D\\
=\big[1-a\alpha^{\pi}(1-\alpha (1+ba))^{-1}bac\big](1+ac)+a(1-ba)^{\#}bac.
\end{array}$$ This completes the proof.\end{proof}

\begin{exam}\end{exam} Let $$\begin{array}{c}
A=\left(
\begin{array}{cccc}
1&0&-1&0\\
0&1&0&-1\\
0&0&0&0\\
0&0&0&0
\end{array}
\right), B=\left(
\begin{array}{cccc}
1&1&1&0\\
0&1&0&0\\
1&-1&0&0\\
0&1&0&0
\end{array}
\right),\\
C=\left(
\begin{array}{cccc}
1&-1&0&1\\
0&1&0&0\\
1&1&0&1\\
0&1&0&0
\end{array}
\right)
\in M_4({\Bbb C}).
\end{array}$$ Then we check that
$$A(BA)^2=ABACA=ACABA=(AC)^2A=0_{4\times 4}.$$ But $ABA=
\left(
\begin{array}{cccc}
0&2&0&-2\\
0&0&0&0\\
0&0&0&0\\
0&0&0&0
\end{array}
\right)\neq
\left(
\begin{array}{cccc}
0&-2&0&2\\
0&0&0&0\\
0&0&0&0\\
0&0&0&0
\end{array}
\right)=ACA$. By directly computation, we have $$\begin{array}{c}
(I_4-AC)^{\#}=\left(
\begin{array}{cccc}
1&-2&0&0\\
0&1&0&0\\
0&0&1&0\\
0&0&0&1
\end{array}
\right),\\
(I_4-BA)^{\#}=\left(
\begin{array}{cccc}
2&3&-1&-3\\
0&2&0&-1\\
1&1&0&-1\\
0&1&0&0
\end{array}
\right).
\end{array}$$

\end{document}